\documentclass[12pt,twoside,reqno]{amsart}
\RequirePackage[english]{babel}
\usepackage[OT1]{fontenc}
\usepackage{type1cm}
\usepackage{subfigure}
\usepackage{enumerate}
\usepackage{amsthm}
\RequirePackage{amsmath,amsfonts,amssymb,amsthm}
\RequirePackage[dvips]{graphicx}
\usepackage{psfrag}
\usepackage{verbatim}
\usepackage[usenames]{color}
\usepackage[dvips]{graphicx}
\usepackage{epstopdf}
\numberwithin{equation}{section}
\usepackage[active]{srcltx}
\newtheorem{thm}{\textbf Theorem}

\newtheorem{defi}{\textbf Definition}

\newtheorem{lem}{\textbf Lemma}[section]
\newtheorem{cor}{\textbf Corollary}[section]
\newtheorem{ex}{\textbf Example}[section]
\newtheorem{que}{\textbf Question}
\newtheorem{rem}{\textbf Remark}

\theoremstyle{remark}
\DeclareSymbolFont{bbold}{U}{bbold}{m}{n}
\DeclareSymbolFontAlphabet{\mathbbold}{bbold}

\begin{document}

\title{Fractal necklaces with no cut points}

\author{\small Fan Wen}

\address{\small Fan Wen\\ College of Mathematics\\ School of Science and Engineering\\ University of Tsukuba\\ Ibaraki 305-8575, Japan\\ s1936006@s.tsukuba.ac.jp}

\bigskip





\maketitle

\begin{abstract}
The fractal necklaces in $\mathbb{R}^d$ ($d\geq 2$) introduced in this paper are a class of connected fractal sets generated by the so-called necklace IFSs, for which a lot of basic topology questions are interesting. We give two subclasses of fractal necklaces and prove that every necklace in these two classes has no cut points. Also, we prove that every stable self-similar necklace in $\mathbb{R}^2$ has no cut points, whilst an analog for self-affine necklaces is false.

\medskip


\medskip


\end{abstract}

\section{Introduction}

Let $I=\{1,2,\cdots, n\}$. For each $k\in I$ let $f_k:\mathbb{R}^d\to\mathbb{R}^d$ be a contractive map satisfying $$|f_k(x)-f_k(y)|\leq c_k|x-y|$$ for all $x,y\in\mathbb{R}^d$, where $c_k\in(0,1)$.  According to Hutchinson \cite{H},
there is a unique nonempty compact subset $F$ of $\mathbb{R}^d$, called the attractor of $\{f_1,f_2,\cdots,f_n\}$, such that
\begin{equation}\label{qq1}
F=\bigcup_{k=1}^nf_k(F).
\end{equation}
We call $\{f_1,f_2,\cdots,f_n\}$ an iterated function system (IFS) of $F$.

\begin{defi}\label{de1}
An attractor $F$ with an IFS $\{f_1,f_2,\cdots,f_n\}$ on $\mathbb{R}^d$ is called a fractal necklace or a necklace for short,
if $n\geq 3$ and $f_k$'s are contractive homeomorphisms of $\mathbb{R}^d$ satisfying
\begin{equation*}
f_m(F)\cap f_k(F)=\left\{
\begin{array}{cc}
\mbox{a singleton} &\mbox{if\, $|m-k|=1$\mbox{ or }$n-1$}\\ \\
\emptyset &\mbox{if\, $2\leq|m-k|\leq n-2$}
\end{array}
\right.
\end{equation*}
for each pair of distinct digits $m, k\in I$. In this case, the ordered family $\{f_1,f_2,\cdots,f_n\}$ is called a necklace IFS or a NIFS. We say that $F$ is self-similar (self-affine), if $f_k$'s are similitudes (affine maps).
\end{defi}

Figure \ref{fig1} illustrates two planar self-similar necklaces. The first one is generated by $3$ similitudes of ratio $1/2$ and the second one is generated by $6$ similitudes of ratio $1/3$. They arise as examples of many papers for various purposes; see for example \cite{RWW,TW}. Among the results of \cite{TW}, Tyson and Wu proved that these two necklaces are of conformal dimension $1$.

\medskip

\begin{figure}[htbp]\label{fig1}
\centering
{\begin{minipage}{4.5cm}
\centering
\includegraphics[width=4.5cm]{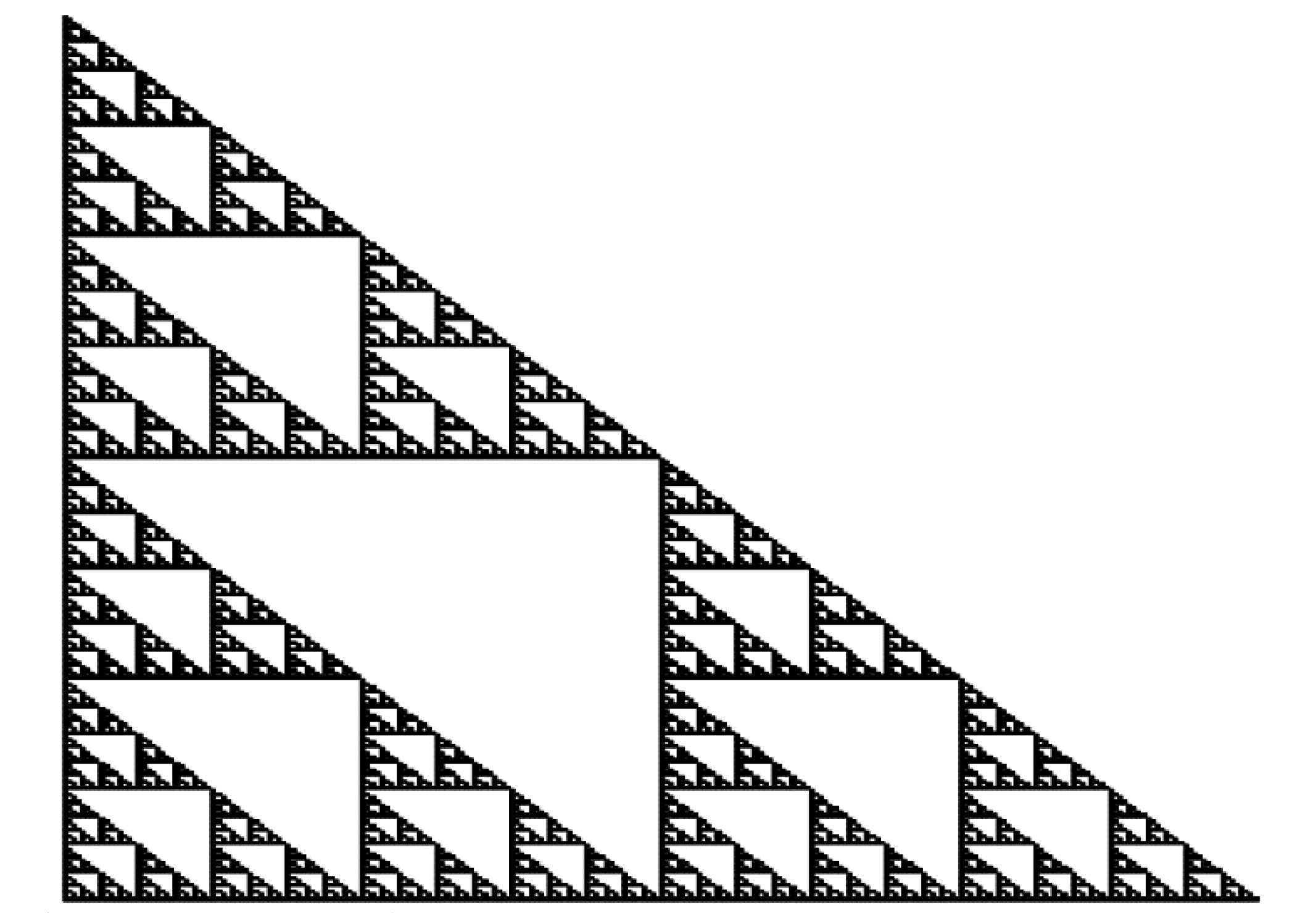}
\end{minipage}}
\quad\quad\quad
{\begin{minipage}{4.5cm}
\centering
\includegraphics[width=4.5cm]{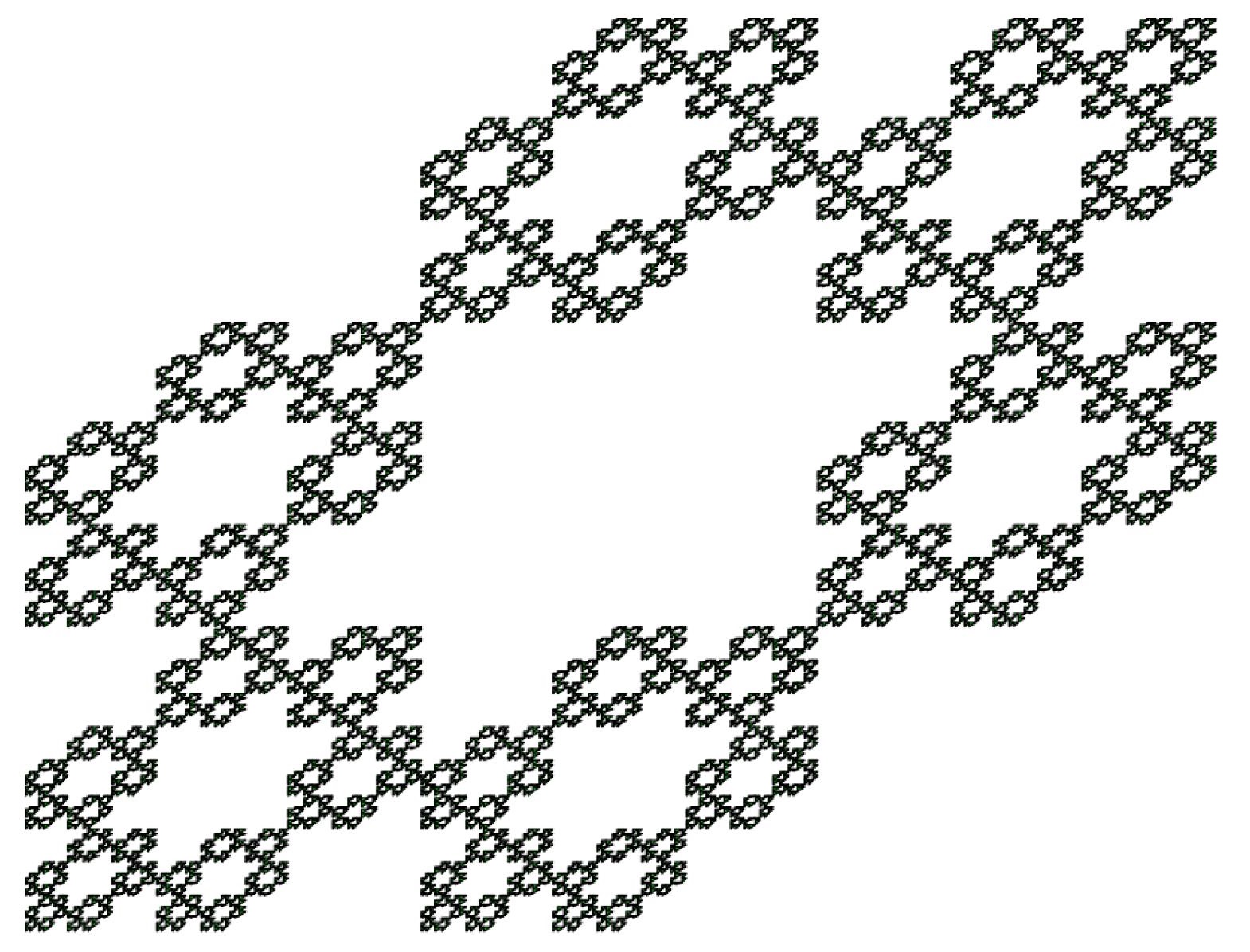}
\end{minipage}}
\caption{Two self-similar necklaces in $\mathbb{R}^2$.}
\end{figure}

It is not difficult to see that every fractal necklace is path-connected and locally path-connected; see \cite{Ha, K}. It is natural to ask whether every necklace has no cut points. The answer is no; see Section \ref{sec2}. Hereafter we say that a point $x$ of a connected topological space $X$ is a cut point, if $X\setminus\{x\}$ is not connected. The present paper is devoted to the following question. For a study on cut points of self-affine tiles we refer to \cite{AL}.

\begin{que}\label{q2}
Which necklaces in $\mathbb{R}^d$ have no cut points?
\end{que}

We start by notations. \emph{From now on denote by $F$ a necklace with a NIFS $\{f_1,f_2,\cdots,f_n\}$ on $\mathbb{R}^d$, if it is not specified.} For every integer $m\geq 0$ and every word $\sigma=i_1i_2\cdots i_m\in I^m$ we write $f_\sigma$ for $f_{i_1}\circ f_{i_2}\circ\cdots\circ f_{i_m}$ and $F_{\sigma}$ for $f_{\sigma}(F)$, where $I^0=\{\emptyset\}$ and $f_\emptyset=id$. The set $F_{\sigma}$ is called an $m$-level copy of $F$. Denote by $\mathcal{C}_m(F)$ the family of $m$-level copies of $F$ and let $$\mathcal{C}(F)=\bigcup_{m=0}^\infty \mathcal{C}_m(F).$$ A copy of $F$ always means a member of $\mathcal{C}(F)$.

For each $k\in I$ denote by $z_k$ the unique common point of the $1$-level copies $F_k$ and $F_{k+1}$. The ordered points $z_1,z_2,\cdots, z_n$ are called main nodes of $F$. For every subset $A$ of $F$ denote respectively by $\mbox{int}_FA$ and $\partial_F A$ the interior and the boundary of $A$ in the relative topology of $F$. Thus $\partial_FF_k=\{z_{k-1},z_k\}$ for every $k\in I$. Hereafter we prescribe
$$F_{n+1}=F_1\,\mbox{ and }\, z_0=z_n.$$

\begin{defi}\label{good}
We say a necklace $F$ is good, if $\partial_FF_k\not\subset F_{kj}$ for any $k,j\in I$.
\end{defi}

Equivalently, a necklace $F$ is good, if $F$ and $F_k$ are the only two copies containing $\partial_FF_k$ for each $k\in I$.

\medskip

Let $I^*=\cup_{m=0}^\infty I^m$ and let $\sigma\in I^*$. Since $f_k$'s have been assumed to be homeomorphisms of $\mathbb{R}^d$, $F_{\sigma}$ is a necklace with an induced NIFS
$$
\{f_{\sigma}\circ f_j\circ f_{\sigma}^{-1}:j\in I\}
$$
whose main nodes are $f_\sigma(z_1), f_\sigma(z_2), \cdots, f_\sigma(z_n)$. The phrase, $m$-level copies of $F_{\sigma}$, is now meaningful.


\begin{defi}\label{stab}
We say a necklace $F$ is stable, if for each $k\in I$
\begin{equation}\label{jq}
\sharp\{F_{kj}: F_{kj}\cap\partial_FF_k\neq\emptyset,\, j\in I\}\geq 2.
\end{equation}
Hereafter $\sharp$ denotes the cardinality.
\end{defi}

By the above definitions, every good necklace is stable. Additionally, every necklace $F$ with the condition
 that $z_{k-1}$ or $z_k$ is a main node of $F_k$ for each $k\in I$ is stable.

\medskip

For each $z\in F$ and for every integer $m\geq 0$ let
\begin{equation}\label{gs}
\mathcal{C}_m(F,z)=\{A\in\mathcal{C}_m(F): z\in A\}
\end{equation}
and $\mathcal{C}(F,z)=\cup_{m=0}^\infty\mathcal{C}_m(F,z)$. Let
\begin{equation}\label{dgs}
c_m(z):=c_m(F,z):=\sharp\,\mathcal{C}_m(F,z)
\end{equation}
denote the number of $m$-level copies containing $z$. Thus $c_1(z)=1$ or $2$, and $c_1(z)=2$ if and only if $z$ is a main node of $F$. Note that for each $A\in\mathcal{C}_m(F, z)$ there is one or two copies $B\in\mathcal{C}_{m+1}(F,z)$ lying in $A$. We have
$$c_m(z)\leq c_{m+1}(z)\leq 2c_m(z).$$
It then follows that $\{c_m(z)\}_{m=1}^\infty$ is a nondecreasing integer sequence satisfying $1\leq c_m(z)\leq 2^m$ for each $m\geq 1$.

\begin{defi}  We say a necklace $F$ is
\emph{of bounded ramification}, if the sequence $\{c_m(z_k)\}_{m=1}^\infty$ is bounded for each $k\in I$.
\end{defi}

Equivalently, a necklace $F$ is of bounded ramification, if $\pi^{-1}(x)$ is finite for any $x\in F$, where $\pi: I^\infty\to F$ is the code map (see \cite{Fa}). By the definition, if there is a main node $z_k$ such that it is a main node of each copy $A\in \mathcal{C}(F, z_k)$, then $F$ is not of bounded ramification. Such necklaces can be found in Figure 3 and Figure 4(b).

\medskip

The main results are as follows.

\begin{thm}\label{gothm}
Every good necklace in $\mathbb{R}^d$ has no cut points.
\end{thm}

\begin{thm}\label{mt1}
Every stable necklace of bounded ramification has no cut points.
\end{thm}

An attractor with an IFS $\{f_1,f_2,\cdots,f_n\}$ on $\mathbb{R}^d$ is said to satisfy the open set condition (OSC), if there is a nonempty bounded open subset $V$ of $\mathbb{R}^d$ such that $f_1(V),f_2(V),\cdots,f_n(V)$ are pairwise disjoint open subsets of $V$; see \cite{Fa,Sc}.

\begin{thm}\label{mp1}
Every stable self-similar necklace in $\mathbb{R}^d$ with the OSC has no cut points.
\end{thm}
Actually, we shall show that every self-similar necklace in $\mathbb{R}^d$ with the OSC is of bounded ramification, which together with Theorem \ref{mt1} implies Theorem \ref{mp1}.

\medskip

As a corollary of a theorem of Bandt and Rao \cite{BR}, every self-similar necklace in $\mathbb{R}^2$ satisfies the OSC.
Thus, Theorem \ref{mp1} gives the following corollary.

\begin{cor}\label{mt2}
Every stable self-similar necklace in $\mathbb{R}^2$ has no cut points.
\end{cor}

\begin{rem}{\rm
A self-similar necklace of bounded ramification in $\mathbb{R}^2$ may have cut points; see Example \ref{e21}.}
\end{rem}

\begin{rem}{\rm
A stable self-affine necklace in $\mathbb{R}^2$ may have cut points; see Example \ref{e22}.}
\end{rem}

\begin{rem}{\rm
Stable necklaces of bounded ramification and good necklaces are not mutually inclusive; see Example \ref{e23}.}
\end{rem}

Without assuming $F$ is self-similar, we have the following result.

\begin{thm}\label{mp2}
Every planar necklace with no cut points satisfies the OSC.
\end{thm}

The paper is organized as follows. In Section \ref{sec2}, we give examples of necklaces to show Remarks 1, 2 and 3. Then we prove Theorem \ref{gothm} in Section \ref{sec3}, Theorem \ref{mt1} in Section \ref{sec4}, and Theorems \ref{mp1} and \ref{mp2} in Section \ref{sec5}. In the light of our results we put some further questions in Section \ref{sec6}.

\section{Examples}\label{sec2}

We first show by an example that a planar self-similar necklace of bounded ramification may have cut points.

\begin{ex}\label{e21}
{\rm We use the complex number notation. Let $\{f_1, \cdots, f_{24}\}$ be a NIFS on the complex plane $\mathbb{C}$ defined by
\begin{equation*}
f_j(z)=\left\{
\begin{array}{cc}
\frac{z}{3}+a_j  &\mbox{if \,} j\in\{1, 7, 13, 19\}\\ \\
\frac{z}{15}+a_j &\mbox{if\, }j\in\{1,2,\cdots, 24\}\setminus\{1, 7, 13, 19\},
\end{array}
\right.
\end{equation*}
where $a_1, a_2, \cdots, a_{24}\in\mathbb{C}$ satisfy
\begin{eqnarray*}
&& f_{24}(1)=f_1\circ f_{13}(i),\,\,f_1\circ f_{13}(1)=f_2(i)\\
&& f_6(1+i)=f_7\circ f_{19}(0),\,\, f_7\circ f_{19}(1+i)=f_8(0)\\
&& f_{12}(i)=f_{13}\circ f_1(1),\,\, f_{13}\circ f_1(i)=f_{14}(1)\\
&& f_{18}(0)=f_{19}\circ f_7(1+i),\,\, f_{19}\circ f_7(0)=f_{20}(1+i)\\
&& f_j(1+i)=f_{j+1}(0),\,\, j\in\{2,4,9,11\}\\
&& f_j(1)=f_{j+1}(i),\,\, j\in\{3,5,20,22\} \\
&& f_j(i)=f_{j+1}(1),\,\, j\in\{8,10,15,17\}\\
&& f_j(0)=f_{j+1}(1+i),\,\, j\in\{14,16,21,23\}.
\end{eqnarray*}
The planar self-similar necklace $F$ generated by $\{f_1, f_2, f_3, \cdots, f_{24}\}$ is illustrated in Figure 2. It has the following properties.
\begin{figure}[htbp]\label{fig2}
\centering
{\begin{minipage}{12cm}
\includegraphics[width=12cm]{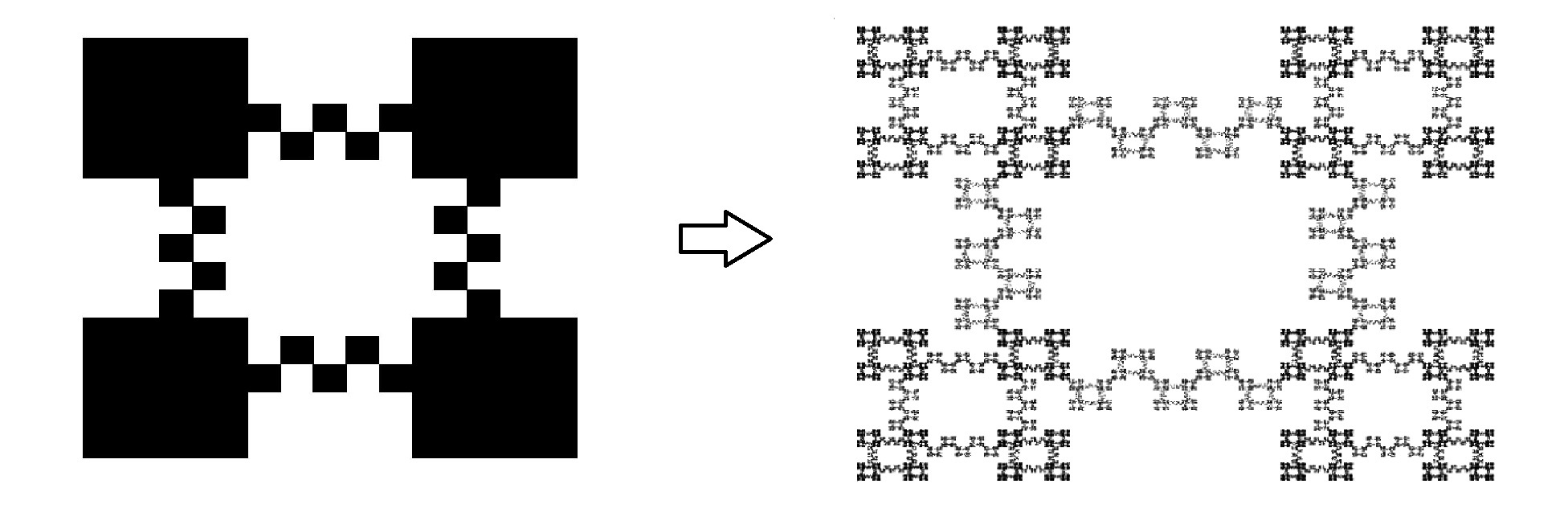}
\end{minipage}}
\caption{A planar self-similar necklace of bounded ramification and with cut points}
\end{figure}

(1) $F$ is not stable, in fact, for the $1$-level copy $F_1$ one has $$\{F_{1j}: F_{1j}\cap\partial_FF_1\neq\emptyset, j\in\{1,2,\cdots, 24\}\}=\{F_{1(13)}\},$$
so $F$ does not satisfy (\ref{jq}). Here the bracketed number in the subscript emphasises that it is a digit.

(2) $F$ is of bounded ramification, indeed, given a main node $z_k$ and an integer $m\geq 1$, $F$ has only two $m$-level copies containing $z_k$.

(3) $(1+i)/4$ is a cut point of $F$. In fact, $F\setminus F_{1(13)}$ is obviously not connected. By zooming we see that $$F\setminus F_{1(13)},\,\,F\setminus F_{1(13)1(13)},\,\, F\setminus F_{1(13)1(13)1(13)},\,\, \cdots,$$ are not connected and tend  to $F\setminus\{(1+i)/4\}$ increasingly, by which one easily shows that $F\setminus\{(1+i)/4\}$ is not connected, as desired.}
\end{ex}

Next we give an example of stable planar self-affine necklaces with cut points.

\begin{ex}\label{e22}
{\rm Let $T_0$ and $T_1$ be two closed solid isosceles triangles sharing a common vertex $z_0$ and of different sizes, whose angles at $z_0$ are a pair of vertical angles and whose opposite sides are parallel. Let $T=T_0\cup T_1$. Let $V$ be the set of the four extremal points of $T$. Let $\{f_1,f_2,\cdots, f_6\}$ be a family of invertible contractive affine maps of $\mathbb{R}^2$ satisfying the following conditions:

1) $V\subset \cup_{k=1}^6f_k(T)\subset T$.

2) $\sharp ( f_j(T)\cap V)=1$ for each $j\in\{1,2,5,6\}$.

3) $V\cap(f_3(T)\cup f_4(T))=\emptyset$.

4) $f_k(T)\cap f_{m}(T)=f_k(V)\cap f_{m}(V)$ if $k\neq m$ and $\{k,m\}\neq\{1,6\}$.

5)$f_1(T)\cap f_6(T)=\{z_0\}$ and $f_1(z_0)=f_6(z_0)=z_0$.

6) $\sharp\, (f_j(T)\cap f_{j+1}(T))=1$ for $j\in\{1,2,3,4,5\}$.

7) $f_k(T)\cap f_{m}(T)=\emptyset$ if $|m-k|\geq 2$.

\noindent Then $\{f_1,f_2,\cdots, f_6\}$ is a NIFS which generates a self-affine necklace $F$ in $\mathbb{R}^2$.  The first step construction of $F$ is illustrated in Figure 3, where the shadow part consists of $f_1(T), f_2(T),\cdots, f_6(T)$. Their connecting points are main nodes of $F$. $z_0$ is a main node and a cut point of $F$.

\begin{figure}[htbp]\label{fig3}
\centering
{\begin{minipage}{10cm}
\centering
\includegraphics[width=10cm]{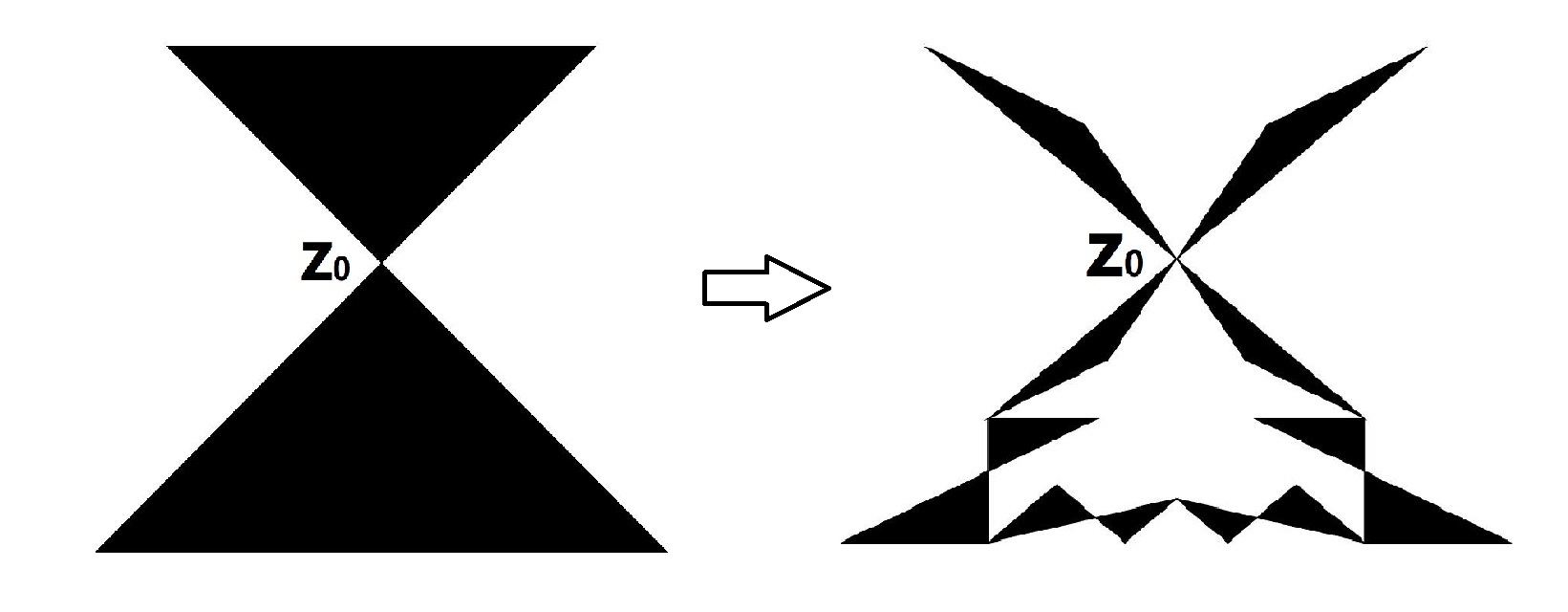}\\
\end{minipage}}
\caption{The first step construction of a stable planar self-affine necklace with cut points.}
\end{figure}

The necklace $F$ is not of bounded ramification, in fact, by condition 5) we have
$$
\mathcal{C}_m(F, z_0)=\{F_{i_1i_2\cdots i_m}: i_1i_2\cdots i_m\in\{1,6\}^m\}\mbox{\, and\, } c_m(z_0)=2^m
$$
for each integer $m\geq 0$, so $\{c_m(z_0)\}_{m=0}^\infty$ is unbounded.

\medskip

The necklace $F$ is stable. In fact, given $k\in\{1,6\}$, we have $z_0\in (\partial_FF_k)\cap F_{k1}\cap F_{k6}$, so
$$
\sharp\{F_{kj}: F_{kj}\cap\partial_FF_k\neq\emptyset,\, j\in \{1,2,\cdots, 6\}\}\geq 2.
$$
On the other hand, given $k\in\{2,3,4,5\}$, we have $\partial_FF_k\subset f_k(V)$ and $\sharp\,(f_k(V)\cap F_{kj})\leq 1$ for each $j\in\{1,2,3,4,5,6\}$, so $\sharp\,((\partial_F F_k)\cap F_{kj})\leq 1$, and so
$$
\sharp\{F_{kj}: F_{kj}\cap\partial_FF_k\neq\emptyset,\, j\in \{1,2,\cdots, 6\}\}\geq 2.
$$
This proves that $F$ is stable.}
\end{ex}

Finally, we show by examples that stable necklaces of bounded ramification and good necklaces are not mutually inclusive.

\begin{ex}\label{e23}
{\rm Let us see Figure 4. The left one is a planar self-similar necklace generated by $\{f_1, f_2, f_3\}$, where $$f_1(z)=\frac{e^{\pi i/6}\overline{z}}{\sqrt{3}},\, f_2(z)=\frac{z+1}{3},\, f_3(z)=\frac{e^{5\pi i/6}z}{\sqrt{3}}+\frac{2}{3},\, z\in\mathbb{C}.$$
Let $F$ be this necklace and let $I=\{1,2,3\}$. Then for each $k\in I$
$$
\sharp\{F_{kj}: F_{kj}\cap\partial_FF_k\neq\emptyset,\, j\in I\}=2,
$$
so $F$ is stable. On the other hand, it is of bounded ramification because
$$c_m(z_1)=c_m(z_2)=3\mbox{\, and \,}c_m(z_3)=2$$
for each integer $m\geq 1$, where $z_1,z_2,z_3$ are main nodes of $F$. In addition, noticing that $F_{13}\supset\partial_F F_1$, we conclude that $F$ is not good.

\begin{figure}[htbp]
\subfigure[]
{\begin{minipage}{5.6cm}
\centering
\includegraphics[width=5.6cm]{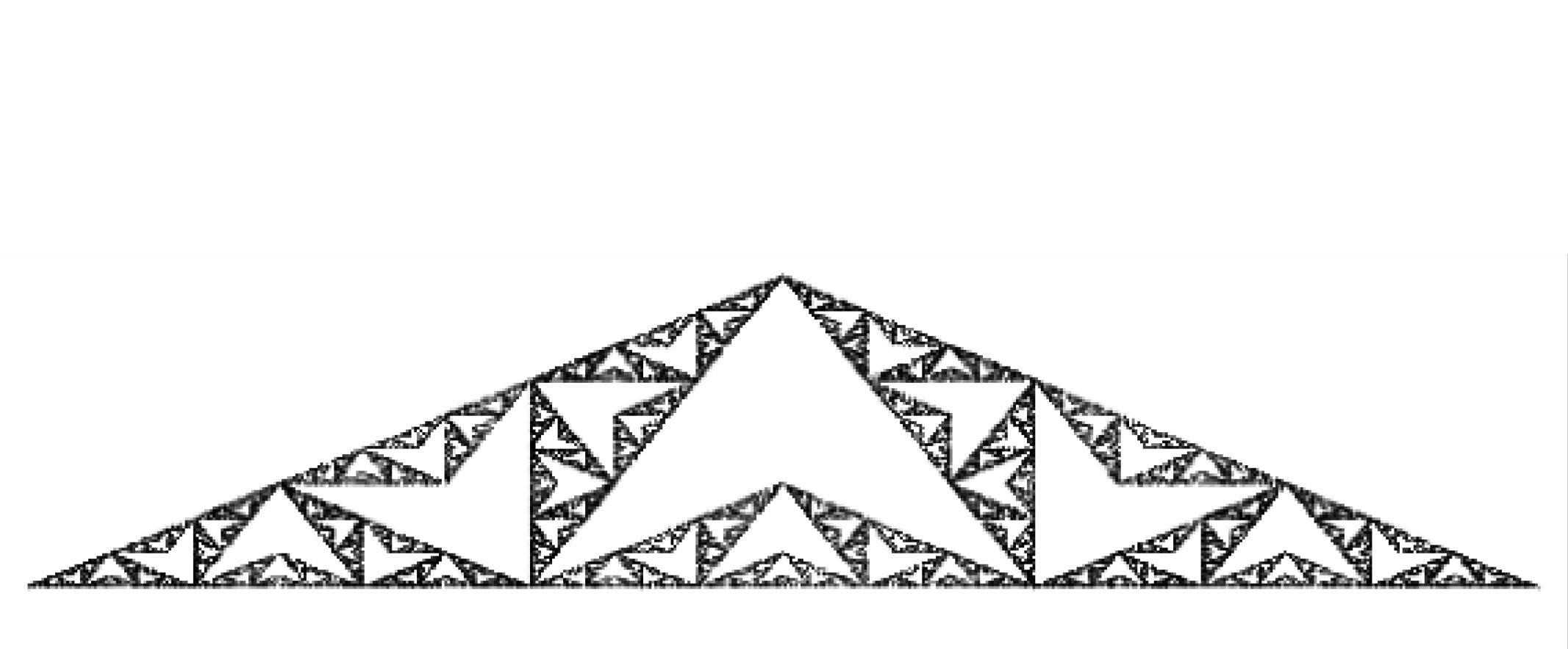}
\end{minipage}}
\quad
\subfigure[]
{\begin{minipage}{5.6cm}
\centering
\includegraphics[width=5.6cm]{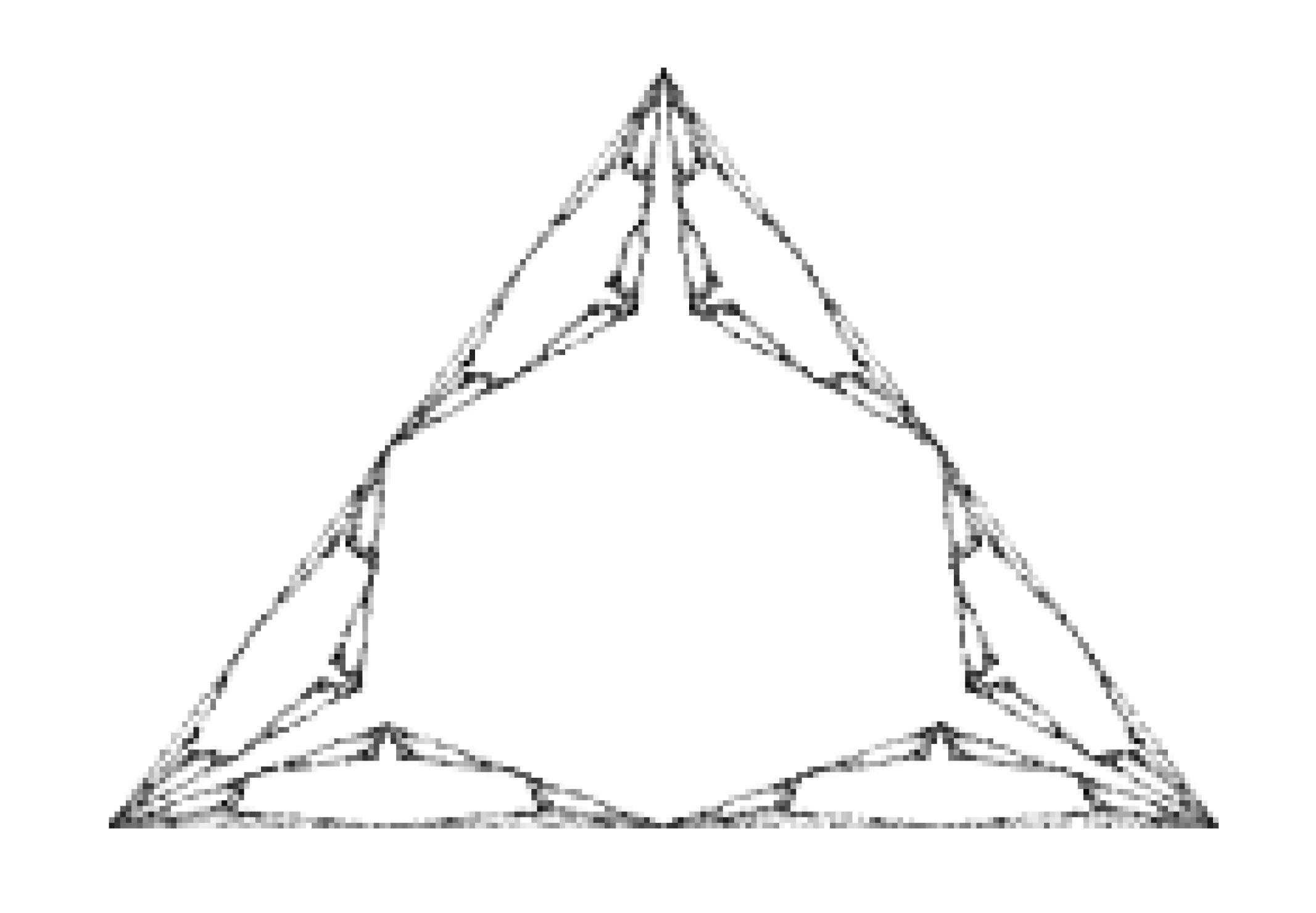}
\end{minipage}}
\caption{Stable necklaces of bounded ramification and good necklaces are not mutually inclusive.}
\end{figure}

Next let $0<\alpha<{\sqrt{3}}/{6}$ be given. For $z=x+iy\in\mathbb{C}$ let
$$g_1(z)=\frac{x+i\alpha y}{2},\,g_2(z)=g_1(z)+\frac{1}{2},$$
$$g_3(z)=e^{\frac{2\pi i}{3}}g_1(z)+1,\,\,g_4(z)=e^{\frac{2\pi i}{3}}g_1(z)+\frac{3+i\sqrt{3}}{4}$$
$$g_5(z)=e^{\frac{\pi i}{3}}\overline{g_1(z)}+\frac{1+i\sqrt{3}}{4},\, \, g_6(z)=e^{\frac{\pi i}{3}}\overline{g_1(z)}.$$
Then $\{g_1,g_2,\cdots,g_6\}$ is a NIFS which generates is a self-affine necklace. It is illustrated on the right of Figure 4.
For each $k\in \{1,2,\cdots, 6\}$ we easily see that $F$ and $F_k$ are the only two copies containing $\partial_FF_k$ and that
$$c_m(z_k)=2^m$$ for each $m\geq 1$. Thus $F$ is good but not of bounded ramification.}
\end{ex}

\section{The Proof of Theorem \ref{gothm}}\label{sec3}

In this section we prove Theorem \ref{gothm}: Every good necklaces has no cut points. The following lemma will be used.

\begin{lem}\label{61}
Let $X$ be a connected metric space, $E$ be connected and dense in $X$, and $x\in X$. If $x$ is a cut point of $X$, then $x$ belongs to $E$ and is a cut point of $E$.
\end{lem}
\begin{proof} Suppose $E\setminus\{x\}$ is connected. Since $E$ is dense in $X$, $E\setminus\{x\}$ is so. It then follows from $E\setminus\{x\}\subset X\setminus\{x\}\subset X$ that $X\setminus \{x\}$ is connected, contradicting the assumption that $x$ is a cut point of $X$. Thus $x$ belongs to $E$ and is a cut point of $E$.
\end{proof}

\medskip

As prescribed, $F$ is a necklace with a NIFS $\{f_1,f_2,\cdots,f_n\}$ on $\mathbb{R}^d$. Let $x, u\in F$ and let $k$ be a positive integer. We say that a finite sequence $(A_1,A_2,\cdots, A_N)$ of $k$-level copies of $F$ is a $k$-level chain from $x$ to $u$, if
$$x\in A_1\setminus \bigcup_{j=2}^NA_j, \,u\in A_N\setminus \bigcup_{j=1}^{N-1}A_j,$$
and
\begin{equation*}
A_j\cap A_m=\left\{
\begin{array}{cc}
\mbox{a singleton} &\mbox{if\, $|j-m|=1$}\\ \\
\emptyset &\mbox{if\, $|j-m|\geq 2$}
\end{array}
\right.\quad j,m\in\{1,2,\cdots, N\}.
\end{equation*}
In this case, we also say that $\cup_{j=1}^NA_j$ is a $k$-level chain. By convention we prescribe $\cup_{j=2}^NA_j=\cup_{j=1}^{N-1}A_j=\emptyset$, if $N=1$.

\medskip

Let $(A_1,A_2,\cdots, A_N)$ be a $k$-level chain of $F$. Denote by $x_j$ the unique point of $A_j\cap A_{j+1}$ for each $j\in\{1,2,\cdots, N-1\}$. We call the ordered points $x_1,x_2,\cdots, x_{N-1}$  the connections of the chain.

\medskip

By Lemma \ref{61}, to prove a topological space has no cut points, it suffices to show that it has a connected dense subset with no cut points. We shall show that every good necklace has Property I and that every necklace with Property I has a connected dense subset with no cut points.
Here we say that a necklace $F$ has Property I, if each of its $1$-level copies $F_k$ has an arc from $z_{k-1}$ to $z_k$ which passes through at least two main nodes of $F_k$. By convention an arc means a subset homeomorphic with the unit interval $[0,1]$.

\medskip

The necklace $F$ in Example \ref{e21} does not have Property I, indeed, $F_1$ does not have a wanted arc with Property I. On the other hand, there are necklaces with  Property I, but they are not good. The left necklace in Figure 4 is one of such.

\medskip

Now we have made the preparations to prove Theorem 1.

\medskip

\noindent{\bf The proof of Theorem \ref{gothm}.} Let $x, u\in F$.

\medskip

{\bf Claim 1.} For each $k\geq 1$, $F$ has a $k$-level chain $\Gamma_k$ from $x$ to $u$. They satisfy $\Gamma_{k+1}\subset\Gamma_k$ and
$$
\{x_1^{(k)}, x_2^{(k)}, \cdots, x_{N_k}^{(k)}\}\subseteq\{x_1^{(k+1)}, x_2^{(k+1)}, \cdots, x_{N_{k+1}}^{(k+1)}\},
$$
where $x_1^{(k)}, x_2^{(k)}, \cdots, x_{N_k}^{(k)}$ are the connections of $\Gamma_k$.

\medskip

{\it Proof.} It is obvious that $F$ has a $1$-level chain from $x$ to $u$.

Suppose $F$ has a $k$-level chain $(A_1,A_2,\cdots, A_N)$ from $x$ to $u$ for an integer $k\geq 1$. In the case where $N=1$, one has $x,u\in A_1$. As is known, $A_1$ has a $1$-level chain from $x$ to $u$. Such a chain of $A_1$ is clearly a $(k+1)$-level chain of $F$ from $x$ to $u$. For the case $N>1$ let $x_1, x_2, \cdots, x_{N-1}$ be the ordered connections of the chain $(A_1, A_2, \cdots, A_{N})$. Then $A_1$ has a $1$-level chain from $x$ to $x_1$, $A_{j}$ has a $1$-level chain from $x_{j-1}$ to $x_{j}$ for each $j\in\{2,3,\cdots, N-1\}$, and $A_N$ has a $1$-level chain from $x_{N-1}$ to $u$. These $N$ chains arranged in the evident order yield a $(k+1)$-level chain of $F$ from $x$ to $u$.

By induction, for each $k\geq 1$, $F$ has a $k$-level chain from $x$ to $u$ with the additional requirements.

\medskip

{\bf Claim 2.} $F$ has an arc from $x$ and $u$.

\medskip

{\it Proof.} For each $k\geq 1$ let $\Gamma_k$ be a $k$-level chain of $F$ from $x$ to $u$ and let $x_1^{(k)}, x_2^{(k)}, \cdots, x_{N_k}^{(k)}$ be its ordered connections as in Claim 1. Let $$\gamma=\bigcap_{k=1}^\infty\Gamma_k.$$
Then $\gamma$ is a compact subset of $F$ containing the connections of $\Gamma_k$ for all $k$. We are going to show that $\gamma$ is an arc from $x$ to $u$.

Let $$X=\bigcup_{k=1}^\infty\{x_1^{(k)}, x_2^{(k)}, \cdots, x_{N_k}^{(k)}\}.$$ Then $X$ is dense in $\gamma$. On the other hand, $X$ is a well ordered set with an ordering induced by those of $\{x_1^{(k)}, x_2^{(k)}, \cdots, x_{N_k}^{(k)}\}$, $k\geq 1$. We may choose a dense subset $$Y=\bigcup_{k=1}^\infty\{y_1^{(k)}, y_2^{(k)}, \cdots, y_{N_k}^{(k)}\}$$ of the interval $[0,1]$ such that the map $h: X\to Y$  defined by $$h(x_j^{(k)})=y_j^{(k)},\, j=1,2,\cdots, N_k, \, k\geq 1$$ is an order-preserving homeomorphism. Now we easily see that $h$ can be extended to a homeomorphism of $\gamma$ onto $[0,1]$.

\medskip

{\bf Claim 3.}  Every good necklace has Property I.

\medskip

{\it Proof.} By the proof of Claim 2, each chain $\Gamma$ of $F$ from $x$ to $u$ has an arc from $x$ to $u$ and such an arc contains the connections of $\Gamma$.

Suppose $F$ is good. To check  Property I, we fix $k\in I$. Note that the connections of every $1$-level chain of $F_k$ are main nodes of $F_k$.

Case 1. Either $z_{k-1}$ or $z_k$ is a main node of $F_k$. Let $\Gamma$ be $1$-level chain of $F_k$ from $z_{k-1}$ to $z_k$. Since $F$ is good, $\Gamma$ contains at least two $1$-level copies of $F_k$, so its connections are nonempty. Let $\gamma$ be an arc of $\Gamma$ from $z_{k-1}$ to $z_k$. Then $\gamma$ contains at least two main nodes of $F_k$.

Case 2. Neither $z_{k-1}$ nor $z_k$ is a main node of $F_k$. In this case, there is a unique pair $l,j\in I$ such that $z_{k-1}\in F_{kl}$ and $z_{k}\in F_{kj}$. Since $F$ is good, we have $l\neq j$.

Subcase 1. $F_{kl}\cap F_{kj}=\emptyset$. Let $\Gamma$ be $1$-level chain of $F_k$ from $z_{k-1}$ to $z_k$. Then $\Gamma$ contains at least three $1$-level copies of $F_k$, so $\Gamma$ has at least two connections. Let $\gamma$ be an arc of $\Gamma$ from $z_{k-1}$ to $z_k$. Then $\gamma$ contains at least two main nodes of $F_k$.

Subcase 2. $F_{kl}\cap F_{kj}\neq\emptyset$. In this subcase, $F_{kl}\cap F_{kj}$ is a singleton whose unique point is denoted by $w$. Let
$$L=\bigcup_{i\in I, i\neq l, i\neq j} F_{ki}.$$
Then $L$ can be regarded as a $1$-level chain of $F_k$ from $a$ to $b$, where $\{a, b\}=\partial_FL$. And we may assume that $a\in F_{kl}$ and $b\in F_{kj}$. Clearly, $a, b$ are main nodes of $F_k$. Since $F$ is good, $F_{kl}$ has a $1$-level chain $A$ from $z_{k-1}$ to $a$ and $F_{kj}$ has a $1$-level chain $B$ from $b$ to $z_{k}$ such that $w\not\in A\cup B$. Thus $A\cap L=\{a\}$, $L\cap B=\{b\}$, and $A\cap B=\emptyset$. Let $\gamma_A$ be an arc of $A$ from $z_{k-1}$ to $a$, $\gamma_L$ be an arc of $L$ from $a$ to $b$, and $\gamma_B$ be an arc of $B$ from $b$ to $z_k$. Then $\gamma_A\cup\gamma_L\cup\gamma_B$ is an arc from $z_{k-1}$ to $z_k$ which contains at least two main nodes of $F_k$.

\medskip

{\bf Claim 4.} Every necklace with  Property I has no cut points.

\medskip

{\it Proof.} Suppose $F$ satisfies  Property I. For each $k\in I$ let $\gamma^{(k)}$ be an arc of $F_k$ from $z_{k-1}$ to $z_k$ which passes through at least two main nodes of $F_k$. Let $$\gamma=\bigcup_{k\in I}\gamma^{(k)}.$$
Then $\gamma$ is a circle of $F$ passing through all main nodes of $F$, where a circle means a subset homeomorphic with the geometric circle.
Let $$E=\bigcup_{\sigma\in I^*} f_{\sigma}(\gamma).$$
Then each $f_\sigma(\gamma)$ is a circle with
\begin{equation}\label{yam1}
f_\sigma(\{z_1, z_2,\cdots, z_n\}\subset f_\sigma(\gamma)\subset F_\sigma
\end{equation} and $E$ is dense in $F$. In addition, by the construction of $\gamma$, we have
$$\sharp(f_\sigma(\gamma)\cap f_{\sigma j}(\gamma))\geq 2$$  for each $\sigma\in I^*$ and each $j\in I$,
from which we easily infer that $E$ is connected and has no cut points. Now, by Lemma \ref{61}, we get that $F$ has no cut points.

This completes the proof of Theorem \ref{gothm}.

\begin{rem}{\rm Let $F$ be a necklace. By Claim 2, $F$ is path-connected. We further conclude that $F$ is locally path-connected, indeed, for each $z\in F$ and each integer $m\geq 1$ the set $$\bigcup_{A\in\mathcal{C}_m(F,z)}A$$ is a path-connected neighborhood of $z$, where $\mathcal{C}_m(F,z)$ is a family of $m$-level copies of $F$ defined by (\ref{gs}).}
\end{rem}

\section{The proof of Theorem \ref{mt1}}\label{sec4}

In this section we prove Theorem \ref{mt1}: Every stable necklace of bounded ramification has no cut points.

\medskip

Let $F$ be a necklace with a NIFS $\{f_1,f_2,\cdots,f_n\}$ on $\mathbb{R}^d$. Let
\begin{equation}\label{orb}
M_F=\bigcup_{\sigma\in I^*}\{f_{\sigma}(z_1),f_{\sigma}(z_2),\cdots, f_{\sigma}(z_n)\}.
\end{equation}
Then $x\in M_F$ if and only if $x$ is a main node of some copy of $F$. Also, we use the notations $\mathcal{C}_m(F,z)$ and $c_m(F,z)$ from (\ref{gs}) and (\ref{dgs}). As each copy $A$ of $F$ is a necklace with an induced NIFS, the notations $M_A$, $\mathcal{C}_m(A,z)$ and $c_m(A,z)$ are self-evident.

\begin{lem}\label{t1}
Suppose $F$ is stable. Then every point of $F\setminus M_F$ is not a cut point of $F$.
\end{lem}

\begin{proof} Fix $z\in F\setminus M_F$. Then, by the definition of $M_F$, for each $m\geq 1$ there is a unique $m$-level copy containing $z$, so $c_m(F,z)=1$ and $z\in\mbox{int}_FV_m$, where $V_m$ denotes the unique member of $\mathcal{C}_m(F,z)$.
Let
\begin{equation}\label{gm0}
U_m=\bigcup_{A\in\,\mathcal{C}_m(F)\setminus\mathcal{C}_m(F,z)}A.
\end{equation}
Then $U_m\cup V_m=F$ and
\begin{equation}\label{kcb}
U_m\cap V_m=\partial_FU_m=\partial_FV_m.
\end{equation}
Furthermore $\{U_m\}_{m=1}^\infty$ is increasing with
\begin{equation}\label{F-x}
F\setminus \{z\}=\bigcup_{m=1}^\infty U_m.
\end{equation}
Let
\begin{equation}\label{L}
L_{m}=\bigcup_{B\in\mathcal{C}_1(V_m)\setminus\mathcal{C}_1(V_m,z)}B.
\end{equation}
Then $L_{m}$ is connected and
\begin{equation}\label{k+1j}
U_{m+1}=U_m\cup L_{m}.
\end{equation}

We claim that $U_m$ is connected for every $m\geq 1$. In fact, $U_1$ is a $1$-level chain of $F$, so it is connected. Assume that $U_m$ is connected for an integer $m\geq 1$. We are going to prove that $U_{m+1}$ is connected.

Since $F$ is stable, we may take two distinct copies $A, B\in\mathcal{C}_1(V_m)$ such that $A\cap\partial_FV_m\neq\emptyset$ and $B\cap\partial_FV_m\neq\emptyset$, so one has $A\cap U_m \neq\emptyset$ and $B\cap U_m \neq\emptyset$ by (\ref{kcb}). Without loss of generality assume $B\neq V_{m+1}$. Then $B\subset L_{m}$ by (\ref{L}). Therefore
\begin{equation}\label{dj}
U_{m}\cap L_{m}\neq\emptyset.
\end{equation}
Since $L_{m}$ is connected and $U_m$ has been assumed to be connected, we get from (\ref{k+1j}) and (\ref{dj}) that $U_{m+1}$ is connected.

By induction, $U_m$ is connected for every $m\geq 1$, which together with (\ref{F-x}) implies that $F\setminus\{z\}$ is connected, so $z$ is not a cut point. This completes the proof.
\end{proof}

\noindent{\bf The proof of Theorem \ref{mt1}.} Suppose $F$ is stable and of bounded ramification. We are going to prove that $F$ has no cut points. As Lemma \ref{t1} was proved, it suffices to prove that every point of $M_F$ is not a cut point of $F$.

Let $z\in M_F$ be given. Then there is a copy $E$ of $F$ such that $z$ is a main node of $E$. In what follows we assume that $E$ is the biggest copy of $F$ with this property. Then $z$ is a main node of $E$ and $z\in\mbox{int}_FE$. To show that $F\setminus\{z\}$ is connected, it suffices to prove that $E\setminus\{z\}$ is connected.

Since $F$ is of bounded ramification, $\{c_m(E,z)\}_{m=1}^\infty$ is bounded.
Thus we may take an integer $k\geq 1$ such that $$c_m(E,z)=c_k(E,z)$$ for all integers $m\geq k$, which in turn implies that $z$ is not a main node of any copy $A\in \mathcal{C}_k(E, z)$.

Therefore $z\in A\setminus M_A$ for each $A\in\mathcal{C}_k(E, z)$.
Since $A$ is stable by the assumption condition, we have by Lemma \ref{t1} that $A\setminus\{z\}$ is connected.

Now that $A\setminus\{z\}$ is connected for each $A\in\mathcal{C}_k(E, z)$, by which we easily see that $B\setminus\{z\}$ is connected for each $B\in\mathcal{C}_{k-1}(E, z)$. Step by step, we get that $E\setminus\{z\}$ is connected. This completes the proof.

\section{The proofs of Theorems \ref{mp1} and \ref{mp2}}\label{sec5}

\noindent{\bf Proof of Theorem \ref{mp1}.}
Let $F$ be a self-similar necklace with a NIFS $\{f_1,f_2,\cdots, f_n\}$ on $\mathbb{R}^d$ satisfying the OSC. We are going to show that $F$ is of boundary ramification.

Let $\mathcal{C}_m(F,z)$ and $c_m(z)$ be defined as (\ref{gs}) and (\ref{dgs}).
We have to show that $\{c_m(z_k)\}_{m=1}^\infty$ is bounded for each main node $z_k$ of $F$.

Given $z_k$ and $m$, let $\tau\in I^m$ be a word such that
$$F_\tau\in\mathcal{C}_m(F,z_k)\mbox{\, and \,}\mbox{diam}(F_{\tau})=\min_{A\in\mathcal{C}_m(F,z_k)}\mbox{diam}(A).$$
For each $A\in \mathcal{C}_m(F,z_k)$ we may take a copy $\widetilde{A}\in\mathcal{C}(F,z_k)$ such that
\begin{equation}\label{po}
\widetilde{A}\subseteq A\, \mbox{ and }\,c_*\mbox{diam}(F_\tau)<\mbox{diam}(\widetilde{ A})\leq \mbox{diam}(F_\tau),
\end{equation}
where $c_*=\min_{1\leq j\leq n}c_j$ and $c_1,c_2,\cdots,c_n\in(0,1)$ are respectively the similarity ratios of $f_1,f_2,\cdots,f_n$. Then we get $c_m(z_k)$ copies of comparable diameters, which are denoted as
\begin{equation}\label{cpis}
F_{\sigma_1}, F_{\sigma_2}, \cdots, F_{\sigma_{c_m(z_k)}}
\end{equation}
where $\sigma_1, \sigma_2, \cdots, \sigma_{c_m(z_k)}\in I^*$ are the corresponding words.  It follows from (\ref{po}) that
$$\bigcup_{j=1}^{c_m(z_k)}F_{\sigma_j}\subset B(z_k,\mbox{diam}(F_\tau)),$$
where $B(z_k,\mbox{diam}(F_\tau))$ is the closed ball of radius $\mbox{diam}(F_\tau)$ centred at $z_k$.
Since the NIFS satisfies the OSC, there is a nonempty bounded open set $V$ of $\mathbb{R}^d$ such that $f_1(V),f_2(V),\cdots,f_n(V)$ are pairwise disjoint subsets of $V$. Thus $$f_{\sigma_1}(V), f_{\sigma_2}(V), \cdots, f_{\sigma_{c_m(z_k)}}(V)$$ are pairwise disjoint.
On the other hand, as $V$ is bounded, we may take a constant $H\geq 1$ such that $V\subset B(z_k, H\mbox{diam}(F))$. Then
\begin{equation}\label{inc}
\bigcup_{j=1}^{c_m(z_k)}f_{\sigma_j}(V)\subset B(z_k,H\mbox{diam}(F_\tau)).
\end{equation}
By comparing volumes we get from (\ref{inc}) that
\begin{equation}\label{RB}
c_m(z_k)\min_{1\leq j\leq c_m(z_k)}\mbox{Vol}(f_{\sigma_j}(V))\leq w_d (H\mbox{diam}(F_\tau))^d,
\end{equation}
where $w_d$ denotes the volume of the $d$-dimensional unit ball. Since the NIFS consists of similitudes of $\mathbb{R}^d$, one has by (\ref{po})
\begin{equation}\label{Sca}
\frac{\mbox{Vol}(f_{\sigma_j}(V))}{\mbox{Vol}(f_{\tau}(V))}
=\left(\frac{\mbox{diam}(f_{\sigma_j}(V))}{\mbox{diam}(f_{\tau}(V))}\right)^d=\left(\frac{\mbox{diam}(F_{\sigma_j})}{\mbox{diam}(F_{\tau})}\right)^d\geq c_*^d
\end{equation}
for every $j\in\{1,2,\cdots, c_m(z_k)\}$,
which together with (\ref{RB}) yields
\begin{equation}\label{RBB}
c_m(z_k)c_*^d\mbox{Vol}(f_\tau(V))\leq w_d (H\mbox{diam}(F_\tau))^d.
\end{equation}
Therefore $$c_m(z_k)\leq \frac{w_d (H\mbox{diam}(F_\tau))^d}{c_*^d\mbox{Vol}(f_\tau(V))}=\frac{w_d(H\mbox{diam}(F))^d}{c_*^d\mbox{Vol}(V)}.$$
Thus the sequence $\{c_m(z_k)\}_{m=1}^\infty$ is bounded. This proves that $F$ is of bounded ramification. Now Theorem \ref{mp1} follows by Theorem \ref{mt1}.

\medskip

\noindent{\bf Proof of Theorem \ref{mp2}.} Let $F$ be a planar necklace with no cut points. We are going to show that $F$ satisfies the OSC.

By the proof of Theorem \ref{gothm}, $F$ has a circle. Thus $\mathbb{R}^2\setminus F$ has infinitely many bounded components by the definition of $F$ and Jordan's curve theorem. Let $U$ be a fixed bounded component of $\mathbb{R}^2\setminus F$.
Then
$$\partial U\subset F\mbox{ and } U\cap F=\emptyset.$$
Let $\{f_1,f_2,\cdots, f_n\}$ be a NIFS of $F$. Since $f_k$'s have been assumed to be contracting homeomorphisms of $\mathbb{R}^2$, we have that, for every $\sigma\in I^*$, the image $f_\sigma(U)$ of $U$ under $f_\sigma$ is a bounded component of $\mathbb{R}^2\setminus F_\sigma$ with
\begin{equation}\label{am0}
\partial (f_\sigma (U))\subset F_\sigma\mbox{ and } f_\sigma(U)\cap F_\sigma=\emptyset.
\end{equation}

We are going to show that  $f_\sigma(U)$ is a bounded component of $\mathbb{R}^2\setminus F$ for each word $\sigma\in I^*$.
First, we have by (\ref{am0}) and the definition of $F$
$$\bigcup_{j=3}^{n-1}F_j\subset f_1(U) \mbox{ or }\left(\bigcup_{j=3}^{n-1}F_j\right)\cap f_1(U)=\emptyset.$$
Moreover, since $\mbox{diam}(F)>\mbox{diam}(F_1)$, we have $$F_2\setminus\overline{f_1(U)}\neq\emptyset \mbox{ or }F_n\setminus\overline{f_1(U)}\neq\emptyset.$$

Next we show that $f_1(U)\cap F=\emptyset$ under the assumption that $F$ has no cut points. In fact, if not, we encounter several different cases
and, in each case, there is a digit $k\in\{2,n\}$ such that $$f_1(U)\cap F_k\neq\emptyset \mbox{ and }F_k\setminus\overline{f_1(U)}\neq\emptyset,$$ which implies that either $z_1$ or $z_n$ is a cut point of $F$, contradicting the assumption on $F$.

Similarly, for each $k\in I$ we have $f_k(U)\cap F=\emptyset$, which together with (\ref{am0}) implies that $f_k(U)$ is a bounded component of $\mathbb{R}^2\setminus F$. For each word $\sigma\in I^*$, arguing as above step by step, we get
$f_\sigma(U)\cap F=\emptyset$, so $f_\sigma(U)$ is a bounded component of $\mathbb{R}^2\setminus F$.

Let $m,k\in I$ be distinct and let $\sigma,\tau\in I^*$. By the above conclusions, $f_m(f_\sigma(U))$ and $f_k(f_\tau(U))$ are two distinct components of $\mathbb{R}^2\setminus F$, so
\begin{equation}\label{os1}
f_m(f_\sigma(U))\cap f_k(f_\tau(U))=\emptyset.
\end{equation}

Now let
$$
V=\bigcup_{\sigma\in I^*}f_\sigma(U).
$$
It is obvious that $f_k(V)\subset V$ for every $k\in I$. On the other hand, we see by (\ref{os1}) that $f_m(V)$ and $f_k(V)$ are disjoint for distinct $m,k\in I$. This proves that, with the open set $V$, the NIFS satisfies the OSC. The proof is completed.

\section{Some further questions}\label{sec6}

\subsection{The OSC problem}

We just proved that every planar necklace with no cut points satisfies the OSC. However, the proof is invalid for necklaces in $\mathbb{R}^d$, $d\geq 3$. Actually, we easily check that every necklace is of topological dimension $1$.  Therefore, for a necklace $F$ in $\mathbb{R}^d$, $d\geq 3$, we see that $\mathbb{R}^d\setminus F$ does not have any bounded components. We do not know if every necklace with no cut points satisfies the OSC in the higher dimensional case. It is open even for self-similar necklaces.

\subsection{Conformal dimension of self-similar necklaces} Tyson and Wu \cite{TW} proved that the two necklaces in Figure 1 are of conformal dimension $1$. We thus ask: Can one develop a unified method to prove that a big class of self-similar necklaces are of conformal dimension $1$?

\subsection{Topological rigidity of necklaces}

Roughly speaking, a topological space is rigid, if the group of its (topological) automorphisms is small.  A further study on good necklaces and their topological rigidity can be found in \cite{Wen}.

\bigskip

\noindent{\bf Acknowledgement.} The author thanks Professors Fang Fuquan and Shigeki Akiyama for their suggestions and encouragement.

\end{document}